\documentclass{article}
\usepackage{amsmath}
\usepackage{multicol}
\usepackage{upgreek}
\usepackage[utf8]{inputenc}\usepackage[T1]{fontenc}
\usepackage{babel}
\usepackage[utf8]{inputenc}
\usepackage{setspace}                   % espaçamento flexível
\usepackage{indentfirst}                % indentação do primeiro parágrafo
\usepackage{makeidx}                    % índice remissivo
\usepackage[nottoc]{tocbibind}          % acrescentamos a bibliografia/indice/conteudo no Table of Contents
\usepackage{courier}                    % usa o Adobe Courier no lugar de Computer Modern Typewriter
\usepackage{type1cm}                    % fontes realmente escaláveis
\usepackage{listings}                   % para formatar código-fonte (ex. em Java)
\usepackage{titletoc}
\usepackage[font=small,format=plain,labelfont=bf,up,textfont=it,up]{caption}
\usepackage[usenames,svgnames,dvipsnames]{xcolor}
\usepackage[a4paper,top=2.54cm,bottom=2.0cm,left=2.0cm,right=2.54cm]{geometry} % margens
\usepackage[plainpages=false,pdfpagelabels,pagebackref=false,colorlinks=true,citecolor=DarkGreen,linkcolor=NavyBlue,urlcolor=DarkRed,filecolor=green,bookmarksopen=true]{hyperref} % links coloridos
\usepackage[all]{hypcap}                % soluciona o problema com o hyperref e capitulos
\fontsize{60}{62}\usefont{OT1}{cmr}{m}{n}{\selectfont}

\usepackage{amssymb,amsthm, amsmath}
\usepackage{empheq}

\usepackage{enumitem}
\usepackage{mathrsfs}

\usepackage{color}

\usepackage[T1]{fontenc}

\usepackage[skip=-5pt]{caption}
\usepackage{mathtools} %usar vcent colon

\usepackage{hyperref}
\usepackage{xcolor}

\usepackage{url}

\newtheorem{theorem}{Theorem}[section]

\newtheorem{theorem-def}[theorem]{Theorem-definition}
\newtheorem{proposition}[theorem]{Proposition}
\newtheorem{corollary}[theorem]{Corollary}
\newtheorem{lemma}[theorem]{Lemma}

\newtheorem{conjecture}[theorem]{Conjecture}

\theoremstyle{definition}
\newtheorem{definition}[theorem]{Definition}
\newtheorem{example}[theorem]{Example}
\newtheorem{remark}[theorem]{Remark}

\usepackage{tikz-cd}

\newcommand{\N}{\mathbb N}
\newcommand{\Z}{\mathbb Z}

\newcommand{\F}{\mathbb F}

\newcommand{\alg}{\mathrm{alg}}
\newcommand{\carac}{\mathrm{char}}
\newcommand{\Mod}{\mathrm{Mod}}
\newcommand{\Tor}{\mathrm{Tor}}
\newcommand{\Ext}{\mathrm{Ext}}
\newcommand{\Hom}{\mathrm{Hom}}
\newcommand{\gldim}{\mathrm{gldim}}
\newcommand{\pd}{\mathrm{pd}}

\title{A Survey on Han's Conjecture} 
\author{Guilherme da Costa Cruz 
\\
Instituto de Matem\'atica e Estat\'istica -- Universidade de S\~ao Paulo
  \\
  {guilhermecc@usp.br}
}
\date{}

\begin{document}

\maketitle

\begin{abstract}
  In 1989, D. Happel pointed out for a possible connection between the global dimension of a finite-dimensional algebra and its Hochschild cohomology: is it true that the vanishing of Hochschild cohomology higher groups is sufficient to deduce that the global dimension is finite? After the discovery of a counterexample, Y. Han proposed, in 2006, to reformulate this question to homology. In this survey, after introducing the concepts and results involved, I present the efforts made until now towards the comprehension of Han's conjecture; which includes: examples of algebras that have been proven to satisfy it and extensions that preserve it.
  
  
\end{abstract}

\section{Introduction}

Cohomology of associative algebras was introduced by G. Hochschild in 1945 \cite{hoch_1945}; just after the same had been made for groups by S. Eilenberg and S. Mac Lane; and some years before cohomology of Lie algebras was brought in by C. Chevalley and S. Eilenberg. After some years,  all of these theories were brought together with a unified approach in H. Cartan and S. Eilenberg's book ``Homological Algebra'', published in 1956 \cite{CartanEilenberg}. This could only be done with a good deal of abstraction -- which was carried out in parallel to the development of Category Theory -- and with the introduction of derived functors. In this manner, Hochschild's cohomology received a new definition through the functor `$\Ext$', and homology was defined dually using `$\Tor$'.

Another important notion introduced in the book was that of projective and global dimension for modules and rings. During the decade of the 1950s, a great deal of research was made in order to understand these concepts and how properties of rings could be understood through them. This led to some significant rewards: for instance, after works of M. Auslander, D. Buchsbaum and J.-P. Serre, many problems concerning regular rings -- which play a fundamental role in Algebraic Geometry -- could be solved. Also concerning homological dimensions, H. Bass published, in 1960, what is now probably the oldest unsolved problem in Homological Algebra: the finitistic dimension conjecture.

By the beginning of 1980's, P. Gabriel had already given a concrete framework for the study of finite-dimensional algebras: quivers (i.e. oriented graphs). He proved that every finite-dimensional algebra could be associated -- without great loss to the study of its modules -- to a quotient of some quiver algebra. Possibly pushed by these results, some interest has risen towards the computation of Hochschild (co)homology for these algebras. This was made clear in an influential paper by D. Happel \cite{happel_89}, in which important previous examples of C. Cibils were also surveyed.

The focus of the present survey resides essentially in a observation made in Happel's article \cite[1.4]{happel_89}: if an algebra has finite global dimension, then it can be proved that its Hochschild cohomology vanishes for higher degrees; what about the converse? An answer to it was given only in 2005, when Buchweitz et al. \cite{happel_falso} published a counterexample. In the meantime, important research was made concerning also Hochschild homology: the vanishing of Hochschild homology was proved to characterize finitude of global dimension for commutative algebras; E. Sk\"{o}ldberg \cite{skoldberg_1999} gave computations for two important quotients of quiver algebras; and others also gave valuable contributions to the understanding of Cyclic Homology -- which is intrinsically related to Hochschild's. Taking all this into consideration, and also after noting that the above counterexample is well behaved when considering its homology, Y. Han \cite[3.4]{han} proposed to reformulate Happel's question to homology, i.e. he conjectured that an algebra has finite global dimension if, and only if, its Hochschild homology vanishes in higher degrees. 

This is where the present survey begins. 

Our main objective is to give a good account on the partial answers already given to Han's conjecture. While some of them can even be deduced from results prior to Han's statement, others were motivated especially by it. This is presented in section \ref{sec:answers}.
To do so, we firstly give a succinct presentation of the notions of global dimension and Hochschild (co)homology of algebras in the preliminaries section \ref{sec:intro}. Afterward, in section \ref{sec:statement}, we establish crucial results providing a proper motivation to the precise statement of Han's conjecture. These are done for arbitrary algebras over a perfect field, in a slight contrast with Han's paper, which is focused in quotient of path algebras. At the final section \ref{sec:fim}, we conclude the paper by making some comments on possible future steps for research. Throughout the paper, we also try to show some subtle aspects in which homology differs from cohomology -- what makes Han's question indeed distinct from Happel's.

In this manner, I hope to provide a clear picture of this topic of research as it is today. This was not done having in mind the specialist solely, in such a way that the beginning graduate student should also feel encouraged to read it -- and invited to the research on the subject. With this in mind, I did not refrain from including references when presenting either a concept that asks for a better introduction or an argument that requires basic results from rings, modules and homology. That said, an acquaintance with some concepts of the theory are desired, such as: simple and semisimple modules; projective and injective modules; complexes and exact sequences; categories and functors; path algebras.

\vspace{0.2cm}
\noindent \textbf{Notation and Terminology:}
Throughout this paper, by an algebra we mean an unital associative algebra over a field. In order to aid the exposition, the reader may also assume that all algebras are noetherian. The word ''two-sided`` will be omitted when talking about two-sided noetherian or artinian algebras, or about two-sided ideals. In addition, the following notations will be used:
\begin{itemize}
    \item $k$ for an arbitrary field;
    \item $k^{\alg}$ for the algebraic closure of $k$;
    \item $A$ and $B$ for $k$-algebras;
    \item $J(A)$ for the Jacobson radical of $A$;
    \item $\otimes$ for the tensor product over $k$, i.e. $\otimes=\otimes_k$;
    \item $A$-$\Mod$ (resp. $\Mod$-$A$) for the category of left (resp. right) $A$-modules
    \item $A^{op}$ for the opposite algebra, i.e. $A$ with multiplication in reverse order.
\end{itemize}

\section{Homology of Associative Algebras}
\label{sec:intro}

We begin by defining the notion of global dimension. As we will see in the example below, one may see it intuitively as a measure on how far an algebra is from being semisimple. For a better understanding on the concept and how it can be used to derive properties of an algebra, I recommend \cite[Sections 4.1-4.4]{weibel_1994}. 

\begin{definition}
Given an $A$-module $M$, its \textit{projective dimension} $\pd_A(M)$ is defined as the minimum $n\in \N$ such that $M$ has a projective resolution of lenght $n$, i.e. an exact sequence
\[0 \rightarrow P_n \rightarrow \hdots \rightarrow P_0 \rightarrow M \rightarrow 0
\]
where each $P_i$ is a projective module. If such a finite resolution does not exist, we write $\pd_A(M)=\infty$. The \textit{global dimension} of $A$ is defined as 
\[
\gldim(A)\vcentcolon =\sup\{\pd_A(M)\mid M\in A\textrm{-Mod}\}.
\]
\end{definition}

\begin{remark}
For a more precise definition, it would be necessary to distinguish the left and right global dimensions of $A$, given when we consider the supremum either over $A$-$\Mod$ or over $\Mod$-$A$. However, as shown by M. Auslander \cite[Corollary 5]{auslander_1955}, they both coincide when $A$ is noetherian.
\end{remark}

\begin{example}

\begin{enumerate}
    \item An algebra $A$ is semisimple if, and only if, every left (or right) $A$-module is projective (see \cite[2.8]{lam_fc}), so $A$ is semisimple precisely when $\gldim(A)=0$.

    \item An algebra $A$ satisfying $\gldim(A)\leqslant 1$ is called \textit{hereditary}. One of the most important examples of these are quiver algebras $kQ$ (also known as path algebras). If its quiver $Q$ does not have oriented cycles, then $kQ$ is finite-dimensional and, in that case, it can be proved that the quotient algebra $kQ/I$ has finite global dimension for any ideal $I$ of $kQ$, cf \cite[Corollary 6]{eilenberg_nagao_nakayama_1956}. For an introduction to path algebras, I refer to \cite[Chapter II]{assem_skowronski_simson_2006}.
    
    \item Noetherian self-injective algebras (also known as quasi-Frobenius) have global dimension equal to zero or to infinity, see \cite[Exercise 4.2.2]{weibel_1994}. This class of algebras contains every Frobenius algebra $A$, i.e finite-dimensional algebras satisfying $A\cong \Hom_k(A,k)$ as $A$-modules, and every symmetric algebra, i.e. the ones satisfying $A\cong \Hom_k(A,k)$ as $A$-bimodules. For a good account on these, I reccomend \cite[Chapter 6]{lam_lec}.

    \item Given a finite-dimensional Lie algebra $\mathfrak{g}$ over a field $k$, the global dimension of its universal enveloping algebra $U\mathfrak{g}$ satisfies
    \[ \gldim(U\mathfrak{g})=\pd_{U\mathfrak{g}}(k)=\dim_k(\mathfrak{g}),
    \]
    cf. \cite[Ex. 7.3.5, Applicaton 7.7.4]{weibel_1994}.
\end{enumerate}
\label{exe:self}
\end{example}

Now, we will give the definition of Hochschild (co)homology in terms of $\Ext$ and $\Tor$ functors. For that, the reader should be aware that an $A$-bimodule $M$ may be considered, equivalently, as a left or right $(A\otimes A^{op})$-module by the following identities:
\[
a\otimes a')\cdot m= ama'=(m\cdot(a'\otimes a),\,\, a,a'\in A, m\in M
\]

\begin{definition}\cite[IX:\S 4]{CartanEilenberg}
The \textit{Hochschild homology groups} (of degree $n$) of an algebra $A$ with respect to a $A$-bimodule $M$ are defined as
\[HH_n(A,M)=\Tor_n^{A\otimes A^{op}}(M,A),\,\,\, n\in\N\]
Its \textit{Hochschild cohomology groups} (of degree $n$) are given by
\[HH^n(A,M)=\Ext^n_{A\otimes A^{op}}(A,M),\,\,\, n\in\N
\]
We shall use the notation $HH_n(A)$ and $HH^n(A)$ for the case $M=A$.
\end{definition}

One can note that each of the abelian groups $HH_n(A,M)$ and $HH^n(A,M)$ have also the structure of a $k$-vector space induced by $A$ and $M$. 

A good introduction to this homological theory is given in \cite{kassel}. For a more thorough study, I refer to \cite[Chapter 9]{weibel_1994} and \cite{loday}. However, we also try to give some intuitions on how this (co)homology behaves. Firstly, we note that the zero degree groups satisfies the following isomorphims:
\[HH_0(A)\cong A/[A,A] \,\,\,\,\,\,\,\,\, 
HH^0(A)=Z(A),
\]
where $Z(A)$ denotes the center of $A$ and $[A,A]=\langle ab -ba \mid a,b\in A \rangle$ is the commutator subspace of $A$. Therefore, zero degree (co)homology measures the commutativity of $A$.

\begin{example}
\begin{enumerate}
    \item If $A=k$, then for any $k$-bimodule $M$ (i.e. a vector space) the (co)homology is trivial:
    \[ HH^n(k,M) \cong HH_n(k,M)\cong
    \begin{cases}k,\,\,\, n=0\\
    0,\,\,\, n> 0
    \end{cases}
    \]
    \item (Truncated polynomial algebras, \cite[5.9]{kassel}) Let $A=k[x]/(p)$ for a polynomial $p$, then the homology groups $HH_n(A)$ are given by the homology of the complex
    \[\hdots \xrightarrow{p'\cdot } A \xrightarrow{0} A \xrightarrow{p'\cdot } A \xrightarrow{0} A \rightarrow 0,
    \]
    where $p'\cdot$ represents the map of multiplication by the (formal) derivative $p'$. We also have $HH_n(A)\cong HH^n(A)$, since $A$ is symmetric \cite[3.15A, 16.55]{lam_lec} (see item \ref{item:symmetric} of proposition below).
\end{enumerate}
\label{exe:polynomials}
\end{example}

Now, we summarize some of the main properties of Hochschild (co)homology.

\begin{proposition} \label{prop:symmetric}
Given algebras $A$ and $B$, we have for each $n\in\N$ that:
\begin{enumerate}
    \item $HH_n(A\times B)\cong HH_n(A)\oplus HH_n(B)$
        
    \item (Change of the ground field) 
    Given an extension of fields $\ell\subseteq k$, the $\ell$-algebra $A_{\ell}=A\otimes\ell$ satisfies $HH_n(A_{\ell})\cong HH_n(A)\otimes\ell$.
    
    \item $HH_{n}(A\otimes B)\cong \bigoplus_{i+j=n} HH_{i}(A)\otimes HH_{j}(B)$
    
    \item \label{prop:morita}
    If $A$ and $B$ are Morita-equivalent (i.e. $A$-$\Mod$ is equivalent to $B$-$\Mod$), then $HH_n(A)\cong HH_n(B)$ .
    
    \item  \label{item:symmetric}
    If $A$ is a finite-dimensional symmetric algebra\footnote{Not to be confused with the symmetric algebra $\textrm{Sym}(V)$ given by a vector space $V$, which is isomorphic to $k[x_1,\hdots,x_n]$ if $\dim_k(V)=n$. Even so, the proposition is unintentionally also valid for these algebras, see \cite[Exercise 9.1.3]{weibel_1994}.}, then $HH_n(A)\cong HH^n(A)$.
\end{enumerate}
Properties $1.$ to $4.$ are also valid for the cohomology groups $HH^{n}(A)$ with the following additional hypothesis for property $3.$: $A$ or $B$ need to be finite-dimensional. 
\end{proposition}
\begin{proof}
    \begin{enumerate}
        \item \cite[Theorem 9.1.8]{weibel_1994}
        
        \item This follows from the following identities:
        \[HH_n(A_{\ell}, A\otimes\ell)\cong HH_n(A, A\otimes\ell)\cong HH_n(A,A)\otimes\ell,
        \]
        where \cite[Theorem 9.1.7]{weibel_1994} was used for the first equality, and that $(-\otimes\ell)$ is an exact functor for the second one, see \cite[Ex. 2.4.2]{weibel_1994}. The same works for cohomology.
        
        \item \cite[Proposition 9.4.1]{weibel_1994} or \cite[4.2.5]{loday}. 
        
        \item \cite[Theorem 9.5.6]{weibel_1994} and \cite[Theorem 2.11.1]{benson_II}. 
        
        \item A symmetric algebra $A$ is characterized by the property $A\cong\Hom_k(A,k)$ as $A$-bimodules. Hence,
        \[HH^n(A)=\Ext^n_{A\otimes A^{op}}(A,A)\cong
        \Ext^n_{A\otimes A^{op}}(A,\Hom_k(A,k))
        \]
        Using \cite[Proposition 2.8.5]{benson_I} and that $k$ is $k$-injective, we deduce that 
        \[HH^n(A)\cong\Hom_k(\Tor^n_{A\otimes A^{op}}(A,A),k)
        =\Hom_k(HH_n(A),k).
        \]
        So, the cohomology groups are the dual spaces of homology ones. Therefore, they are isomorphic when $A$ is finite-dimensional. 
    \end{enumerate}
    \vspace{-14pt}
\end{proof}

\begin{remark}
It is worth mentioning that a generalization of item \ref{prop:morita} was proved by D. Happel in the framework of finite-dimensional algebras \cite[4.2]{happel_89} -- namely, that cohomology of $A$ and $B$ are equal if $B$ is ``tiltable'' to $A$. This was shown in a more general setting (including any algebra over a field) by J. Rickard \cite[Proposition 2.5]{rickard}, soon after giving a more profound characterization on the tiltable property for any rings \cite[Theorem 1.1]{rickard}. In more detail, he proved that $B$ is tiltable to $A$ if, and only if, its derived categories are equivalent -- and in that case we say that $A$ and $B$ are \textit{derived equivalent}. Similarly, it can be proved that Hochschild homology is invariant by derived equivalences, cf. \cite[Theorem 2.2]{keller1996under}.
\end{remark}

The following example shows how these properties may be valuable in order to calculate Hochschild (co)homology of an algebra.

\begin{example}
Given a finite-dimensional semisimple algebra $A$ over an algebraically closed field $k$, we know by the Wedderburn-Artin theorem that 
\[A\cong \bigoplus_{i=1}^m M_{n_i}(k)
\]
for some $n_i, m\in\N$. So, using properties 1 and \ref{prop:morita} and that $M_n(k)$ is Morita-equivalent to $k$, we may conclude that
\[HH_n(A)\cong HH^n(A) \cong
\begin{cases}k^{m}, n=0\\
    0, n> 0
    \end{cases}.
\]
\end{example}

Now, we give an important representative for the Morita-equivalence class of an algebra.

\begin{theorem}
Assume $k$ is an algebraically closed field, then every finite-dimensional algebra is Morita-equivalent to an (admissible) quotient of a path algebra $kQ/I$.
\label{teo:Gabriel}
\end{theorem}
\vspace{-8pt}
\begin{proof}[Comments on the proof]
This follows from two facts: 
\begin{itemize}
    \item Every finite-dimensional algebra is Morita-equivalent to a basic algebra \cite[18.37]{lam_lec}.
    \item Every basic algebra is isomorphic to an admissible quotient of a path algebra.
\end{itemize}
Over an algebraically closed field, the second item is a well-known result of P. Gabriel, see \cite[section II.3]{assem_skowronski_simson_2006}. 
A similar result may also be proved if one considers, more generally, perfect fields (definition \ref{def:perfect}). An outline for the proof can be found in \cite[Corollary 4.1.11]{benson_I} and, for a more detailed approach, see \cite[Theorem 3.12]{Gabriel_perfeito}, where the proofs are carried out by using the notion of species.
\end{proof}

For this reason, when studying Hochschild (co)homology of finite-dimensional algebras, not much generality is lost if one considers just quotients of path algebras -- and that is what many authors do (e.g. D. Happel and Y. Han).

Now, we mention two properties that are valid exclusively for homology.

\begin{proposition}\label{prop:corner}
For each $n\in\N$, we have that:
\begin{enumerate}
    \item $HH_n(-)\colon \mathrm{Alg}_k\to \mathrm{Vect}_k$ is a functor from the category of $k$-algebras to the category of $k$-vector spaces.
    
    \item Given algebras $A$ and $B$ and a $A$-$B$-bimodule $M$,
\[HH_n(\begin{bmatrix}
A & M\\
0 & B
\end{bmatrix}) \cong HH_n(A) \oplus HH_n(B).
\]
\end{enumerate}
\end{proposition}
\begin{proof}
\cite[1.1.4]{loday} and \cite[1.2.15]{loday}
\end{proof}
\begin{remark}
The first property is not valid, for example, in zero degree cohomology: the center $Z(-)$ is not a functor.
\end{remark}

As shown in the example above, Hochschild cohomology of matrix algebras over $k$ vanishes for every $n$ and every bimodule. In what follows, we provide some characterizations of algebras satisfying this property.

\begin{theorem-def}
We say that an algebra $A$ is \textit{separable} if it satisfies the following equivalent conditions:
\begin{enumerate}
    \item $HH^i(A,M)=0$ for every $i>0$ and every $A$-bimodule $M$.  \label{sep_1}
    \item $A\otimes A^{op}$ is semisimple.\label{sep_2}
    \item $A$ is finite-dimensional and $A\otimes \ell$ is semisimple for every field extension $\ell\supseteq k$.\label{sep_3}
    \item $A$ is finite-dimensional and $A\otimes k^{\alg}$ is semisimple.\label{sep_4}
\end{enumerate}
\label{teo:separable}
\end{theorem-def}
\vspace{-10pt}
\begin{proof}[Comments on the proof]
The equivalence $\ref{sep_1}\Leftrightarrow \ref{sep_3}$ was already proved in G. Hochschild's 1945 paper \cite[Theorem 4.1]{hoch_1945}, showing how his cohomology can be a useful tool to understand properties of associative algebras. More modern proofs may be found in \cite[IX: Theorems 7.9, 7.10]{CartanEilenberg} and in \cite[Theorem 9.2.11]{weibel_1994}.
\end{proof}
\vspace{-8pt}

As it can be seen by the characterization \ref{sep_3}, every separable algebra is semisimple. So, one may ask when the converse holds. As we will show below, the answer is to consider perfect fields. This good behaviour is one of the main reasons that many of the results in the next section will be formulated over fields of this class, which is not a small one: it includes fields that are either finite, algebraically closed or of characteristic zero.

\begin{definition}
A field $k$ is said to be \textit{perfect} if every finite (or algebraic) extension of $k$ is separable.
\label{def:perfect}
\end{definition}

We recall that an algebraic extension $\ell\supset k$ is \textit{separable} if, and only if, for every $\alpha\in \ell$ the derivative of the minimal polynomial of $\alpha$ over $k$ is non-zero. This is consistent with the above notion of separable algebras: a finite extension $\ell\supset k$ is separable if, and only if, $\ell$ is a separable $k$-algebra \cite[9.2.8]{weibel_1994}.

\begin{proposition}
A field $k$ is perfect if, and only if, every finite-dimensional semisimple $k$-algebra is separable.
\label{prop:perfect}
\end{proposition}
\vspace{-8pt}
\begin{proof}
A finite-dimensional algebra $A$ is semisimple if, and only if, $J(A)=0$. Furthermore, $J(A\otimes\ell)=J(A)\otimes\ell$ for every separable algebraic extension $\ell\supseteq k$, cf. \cite[5.17]{lam_fc}. So, if $k$ is a perfect field, we have that $J(A\otimes k^{\alg})=0$ for every semisimple algebra $A$. The converse follows immediately from the definition: if $k$ is not perfect, then there exists a field $\ell$ which is finite-dimensional over $k$ and is not separable.
\end{proof}

\section{Statement of Han's conjecture}
\label{sec:statement}

In this section, restricting ourselves to finite-dimensional algebras $A$, we will show that, if $\gldim(A)$ is finite, then its Hochschild homology is concentrated solely in degree zero. In this manner, we will get a legitimate motivation for the statement of Han's conjecture. Before that, we will prove a more elementary result, and which is also valid for cohomology. 

In what follows, we will use the following standard notation:

\begin{definition}
The \textit{Hochschild homological} (resp. \textit{cohomological}) \textit{dimension} of an algebra $A$ is defined as 
\begin{align*}
    \mathrm{hh.dim}(A)\vcentcolon =& \sup\{n\in\N\mid HH_n(A)\neq 0\}\\
    \mathrm{hch.dim}(A)\vcentcolon =& \sup\{n\in\N\mid HH^n(A)\neq 0\}.
\end{align*}
If, by any chance, $HH_n(A)=0$ (resp. $HH^n(A)=0$) for all $n$, we settle, as a convention, that $\mathrm{hh.dim}(A)=0$ (resp. $\mathrm{hch.dim}(A)=0$).
\end{definition}

In the following results, we will assume that $A/J(A)$ is separable, which is always true when $k$ is a perfect field. Indeed, this follows by proposition \ref{prop:perfect} and the fact that $A/J(A)$ is semisimple.

\begin{proposition}\label{prop: hoch_finita}
If $A$ is a finite-dimensional algebra such that $A/J(A)$ is separable (e.g. $k$ is a perfect field), then:
\begin{enumerate}
    \item $\gldim(A\otimes A^{op})=2\cdot\gldim(A)$.
    \item $\gldim(A\otimes\ell)=\gldim(A)$ for every field extension $\ell\supseteq k$.
\end{enumerate}
\end{proposition}
\vspace{-6pt}
\begin{proof}
Using the notation $\overline{A}=A/J(A)$, it follows from \ref{teo:separable} that:
\begin{enumerate}
    \item $\overline{A}\otimes \overline{A^{op}}$ is semisimple;
    \item $\overline{A}\otimes\ell$ is semisimple for every field extension $\ell\supseteq k$.
\end{enumerate}
In this way, the proposition follows from a result of Auslander \cite[Theorem 16]{auslander_1955}.
\end{proof}

%\begin{remark}
%The second item may also be retrieved from \cite[Theorem 2.4]{trocar_corpo}, and from knowing that every extension of a perfect field is (MacLane-)separable.
%\end{remark}

From the definition of $\Ext$ and $\Tor$ functors, it is possible to conclude that
\[\mathrm{hh.dim(A)},\mathrm{hch.dim(A)} \leqslant \pd_{A\otimes A^{op}}(A) \leqslant \gldim(A\otimes A^{op}).
\]
In this manner, we obtain the following consequence from the first item\footnote{We could also use the following result: $\pd_{A\otimes A^{op}}(A) = \gldim(A)$, see \cite[\S 4]{eilenberg_nagao_nakayama_1956}.}:

\begin{corollary}
Every finite-dimensional algebra $A$ such that $A/J(A)$ is separable (e.g. $k$ is a perfect field) satisfies: 
\[\gldim(A)< \infty \implies \mathrm{hh.dim}(A)< \infty,\,\, \mathrm{hch.dim}(A)< \infty.
\]
\label{coro: hoch_finita}
\end{corollary}

\begin{remark}
In the results above, the hypothesis over the field is truly necessary: if $k$ is not a perfect field, then it has a non-separable element $\alpha\in k^{\alg}\setminus k$, so that its minimal polynomial $m_{\alpha}$ has zero derivative. Therefore, $k(\alpha)=k[x]/(m_{\alpha})$ is a finite-dimensional $k$-algebra with $\gldim(k(\alpha))=0$ (since it is a field) whose Hochschild (co)homology is, by example \ref{exe:polynomials}, always non-zero:
\[HH^n(k(\alpha))\cong HH_n(k(\alpha))\cong k(\alpha)
\]
for every $n\geqslant 0$. For a concrete example, one can take $k=\mathbb{F}_p(t)$ and $\alpha=\sqrt[p]{t}$ for some prime $p$, so that $m_{\alpha}=x^p -t$.
\end{remark}

Now, we will see that we have a much stronger result for the homological dimension, which is essentially a consequence of the following result by B. Keller.

\begin{lemma} \textup{\cite[2.5]{KELLER1998223}}
Suppose $A$ is a finite-dimensional algebra such that $\overline{A}=A/J(A)$ is a product of copies of $k$ and $\Hom_A(S,S)\cong k$ for each simple $A$-module $S$.\footnote{The second assumption can be proved to be superfluous, see \cite[4.8, 7.7]{lam_fc}} If $A$ has finite global dimension, then we have an isomorphism (induced by the inclusion $\overline{A}\hookrightarrow A$) of the cyclic homology groups $HC_n(\overline{A})\cong HC_n(A)$ for every $n\geqslant 0$.
\label{lemma:keller}
\end{lemma}

The reader not acquainted with Cyclic Homology should not be alarmed by its use in the formulation of the above. The cyclic homology groups have an intrinsic relation with Hochschild ones, given by the so-called Connes' long exact sequence:
\[\hdots \rightarrow HC_{n+1}(A)\rightarrow HC_{n-1}(A)
\rightarrow HH_n(A)\rightarrow HC_n(A)\rightarrow HC_{n-2}(A)
\rightarrow \hdots
\]
For instance, when $n=0$, we have the isomorphism $HC_0(A)\cong HH_0(A)$. Furthermore, as we note below, we could have replaced $HC$ by $HH$ when writing the lemma.

\begin{lemma}\textup{\cite[$2.2.3$]{loday}}
    Let $f\colon A\to A'$ be a morphism of $k$-algebras.
    \[
    f \text{ gives the isomorphism } HH_*(A)\cong HH_*(A')
    \iff 
    f \text{ gives the  } HC_*(A)\cong HC_*(A')
    \]
    \label{lema:hc=hh}
    \vspace{-0.8cm}
\end{lemma}

From these results, we obtain the following synthesis:

\begin{theorem}[Keller]
Every finite-dimensional algebra $A$ such that $A/J(A)$ is separable (e.g. $k$ is a perfect field) satisfies: 
\[\gldim(A)< \infty \implies \mathrm{hh.dim}(A)=0.
\]
\end{theorem}

\begin{proof}
We fix the notation $A_{k^{\alg}}=A\otimes k^{\alg}$. Using that $A/J(A)$ is separable, we deduce that $\overline{A_{k^{\alg}}}= A_{k^{\alg}}/J(A_{k^{\alg}}) $ is semisimple (over an algebraically closed field), so that it is isomorphic to a direct sum of matrix algebras $M_n(k^{\alg})$ by the Wedderburn-Artin theorem. Hence, the associated basic algebra $(A_{k^{\alg}})^b$, which is Morita-equivalent to $A_{k^{\alg}}$, satisfies the hypothesis of lemma \ref{lemma:keller}.

Now, from \ref{prop: hoch_finita}, we also know that $\gldim(A)=\gldim(A_{k^{\alg}})$. In this manner, applying both lemmas (and some Morita invariance), we get that:
\[\gldim(A)<\infty 
\implies 
HH_n(A_{k^{\alg}})\cong HH_n((A_{k^{\alg}})^b)\cong
HH_n\Big(\dfrac{(A_{k^{\alg}})^b}{J(A_{{k}^{\alg}})^b}\Big)
\]
Since the quotient is a product of copies of $k^{\alg}$, we can conclude that the latter is zero for every $n>0$. Finally, this is also valid for $HH_n(A)$, since
\[HH_n(A_{k^{\alg}}) \cong HH_n(A)\otimes {k^{\alg}}\,\,\, \textrm{ for all } n\geqslant 0
\]
by property 2 in proposition \ref{prop:symmetric}.
\end{proof}

Now, we finally formulate Han's conjecture for perfect fields, which is basically  the converse of the results proved above.

\begin{conjecture}[Han] If $A$ is a finite-dimensional algebra such that $A/J(A)$ is separable (e.g. $k$ is a perfect field), then the following are equivalent:
\begin{enumerate}
    \item $\mathrm{hh.dim}(A)<\infty $\label{item_1}
    \item $\mathrm{hh.dim}(A)=0$
    \item $\gldim(A)<\infty $\label{item_3}
\end{enumerate}
\end{conjecture}

Since implication \ref{item_1} $\Rightarrow$ \ref{item_3} is the only one that hasn't been proved yet, and it is also studied for algebras in general, we establish:

\begin{definition}
The implication $\mathrm{hh.dim}(A)<\infty\implies \gldim(A)<\infty$ is called \textit{Han's property}. The analogous statement for cohomology (i.e $\mathrm{hch.dim}(A)<\infty\implies \gldim(A)<\infty$) is called \textit{Happel's property}.
\end{definition}

As one might wonder, Keller theorem does not have an analogous for cohomology. In fact, using Happel's computation in \cite[1.6]{happel_89}, one can take even path algebras (over an algebraically closed field) as counterexamples: if $Q$ is a quiver (without oriented cycles) whose underlying graph is not a tree, then $HH^1(kQ)\neq 0$. For example, taking $Q_1= (1 \rightrightarrows 2)$ and $Q_2$ to be
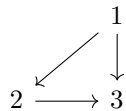
\[
\begin{tikzcd}
& 1 \arrow[d] \arrow[dl] &   \\
2   \arrow[r] & 3 
\end{tikzcd}
\]
we get that $HH^{1}(kQ_1)\cong k^{3}$ and $HH^{1}(kQ_2)\cong k$.

\section{(Partial) Answers to Han's Conjecture}
\label{sec:answers}

\subsection{Algebras satisfying Han's or Happel's property}
\label{sec:examples}

We summarize in tables \ref{table:Han's examples} and \ref{table:Happel's examples} below, the classes of algebras which have been proven to satisfy, respectively, Han's and Happel's property. In what follows, some comments will be made in order to help with two types of difficulties when reading it. The first one is that many of the examples are not exactly well-know -- and some have been defined only in the reference paper -- so we provide the definitions for some of these cases. In a second aspect, it may not be clear the reason why Han's property follows from the theorems in the references, thus some clarifications are given in this direction.

As we will note below, two of these classes (group algebras and trivial extensions) are of symmetric algebras, so both properties are equivalent by proposition \ref{prop:symmetric}. With this in mind, even though they also satisfy Happel's property, we have recorded them only in Han's table.

\begin{table}[htb]
\centering
\renewcommand{\arraystretch}{1.4}
\begin{tabular}{|lll|} 
\hline
\multicolumn{1}{|l|}{\textbf{Class of algebras}}                                              & \multicolumn{1}{l|}{\textbf{Assumption over the field}}                      & \textbf{References}                      \\ 
\hline
group algebras                                                                                & -~                                                                            &
\cite{swan}, \cite[Thm I.1]{Burghelea1985}
             \\ 
\hline

\begin{tabular}[c]{@{}l@{}}  quotients of \vspace{-6pt}\\  acyclic quiver algebras\end{tabular}       & -~                                                                            &
\cite[Cor. 6]{eilenberg_nagao_nakayama_1956}, \cite{cibils_86}               \\ 
\hline
commutative                                                                          & -                                                                             & \cite{vigue_92}, \cite{buenos_aires}                    \\ 
\hline
exterior algebras                                                                             & -                                                                             & \cite[Theorem 2]{exterior}                      \\ 
\hline
monomial & -                                                                             & \cite[Theorem 3]{han}                             \\
\hline
quantum complete intersections      & -                                                & \cite[Theorem 3.1]{bergh2008quantum}                  \\ 
\hline
$N$-Koszul                                                                              & $\carac(k)=0$                                                                 & \cite[Theorem 4.5 ]{bergh2009hochschild}                       \\ 
\hline
\begin{tabular}[c]{@{}l@{}}homogeneous quotients of \vspace{-6pt} \\quiver algebras with loops\end{tabular} & $\carac(k)=0$                                                                 &\cite[Theorem 4.7]{bergh2009hochschild}                       \\ 
\hline
graded cellular                                                                               & $\carac(k)=0$                                                                 & \cite[Theorem 4.9 ]{bergh2009hochschild}                      \\ 
\hline
\begin{tabular}[c]{@{}l@{}}a generalization of quantum \vspace{-6pt} \\complete intersections\end{tabular}  & -                                                                             & \cite[Theorem I   ]{solotar2010two}             \\ 
\hline
\begin{tabular}[c]{@{}l@{}}  local graded algebras\vspace{-6pt}\\  with a certain relation\end{tabular} & -                                                  & \cite[Theorem II]{solotar2010two}                 \\ 
\hline
\begin{tabular}[c]{@{}l@{}}  quantum generalized \vspace{-6pt}\\  Weyl algebras\end{tabular}                                                             & $\carac(k)=0$                                                                 & 
\cite[1.1, 1.2, 3.3]{solotar_Weyl}
 \\ 
\hline
trivial extensions of local~                                                                  & algebraically closed                                                          & \cite[Theorem 3.2]{trivial_extension}
\\ 
\hline
\begin{tabular}[c]{@{}l@{}}  trivial extensions \vspace{-6pt}\\  of self-injective \end{tabular}  & algebraically closed                                             & \cite[Theorem 3.5]{trivial_extension}                              \\ 
\hline
trivial extensions of graded                                                                  & \begin{tabular}[c]{@{}l@{}}algebraically closed, \vspace{-6pt}\\ $\carac(k)=0$\end{tabular} & \cite[Theorem 3.9]{trivial_extension}                         \\
\hline
\end{tabular}
\vspace{0.4cm}
\caption{Known examples of algebras satisfying Han's property. With the exception of lines 3, 10 and 12, all of them are assumed to be finite-dimensional. The list is organized in chronological order of the references.}
\label{table:Han's examples}
\end{table}

\begin{table}[htb]
\centering
\renewcommand{\arraystretch}{1.4}
\begin{tabular}{|ll|} 
\hline
\multicolumn{1}{|l|}{\textbf{Class of algebras}} & \textbf{Reference}                                   \\ 
\hline
commutative                                      & \cite[Corollary]{Happel_comutativa}                   \\ 
\hline
exterior algebras          & \cite[Theorem 3]{exterior}        \\ 
\hline
truncated                & \cite[Theorem 3]{Happel_truncada}  \\ 
\hline
some quantum complete intersections      & \cite[Theorem 3.3]{bergh2008quantum}                 \\
\hline
 quantum generalized Weyl algebras & \cite[Theorems 1.1, 1.2, 3.3]{solotar_Weyl}  \\ 
\hline
\end{tabular}
\vspace{0.4cm}
\caption{Examples of algebras satisfying Happel's property. With the exception of the last class, all of them are assumed to be finite-dimensional over an arbitrary field. The list is organized in chronological order of the references.}
\label{table:Happel's examples}
\end{table}

\vspace{0.2cm} \noindent \textbf{Group Algebras:}
In this section, all groups are assumed to be finite.
The fact that every group algebra satisfies Han's property was not explicitly found in the literature. However, the following proof, which was essentially communicated by Eduardo N. Marcos, is easily deduced from somewhat well-know facts from Group (Co)Homology. 

First of all, one should be aware that every group algebra is symmetric \cite[16.56]{lam_lec}, so that its Hochschild homology and cohomology are isomorphic. Another important aspect is that its global dimension have only two possible values: zero or infinite. By Maschke's theorem, we know that a group algebra $kG$ has zero global dimension if, and only if, $\carac(k)$ does not divide the order of $G$. In this manner, the assertion that $kG$ satisfies Han's property is equivalent to the following:

\begin{theorem}
If $\carac(k)=p>0$ divides the order of a finite group $G$, then $\mathrm{hh.dim}(kG)=\infty$.
\label{teo:grupo}
\end{theorem}

Now, a result of Burghelea \cite[Theorem I.1]{Burghelea1985} shows that homology of group algebras can be computed in terms of Group Homology. The same holds for cohomology, see \cite[Theorem 2.11.2]{benson_II}. These results show, in particular, that the (co)homology groups of $G$ with respect to $k$, denoted by $H_n(G,k)$ and $H^n(G,k)$\footnote{In terms of $\Ext$-$\Tor$ functors, they can defined as  $H_n(G,k)=\Tor^{kG}_n(k,k)$ and 
$H^n(G,k)=\Ext_{kG}^n(k,k)$}, are direct summands, respectively, of $HH_n(kG)$ and $HH^n(kG)$. Thus, the proof can be concluded by using a result of R. Swan \cite{swan}: it guarantees that, if $\carac(k)=p$ divides the order of $G$, then $H^n(G,k)$ is non-zero for an infinite number of values of $n>0$. 

One final comment about Swan's article should be made: although it is focused in cohomology with coefficients in $\Z$, the author also remarks that his arguments are also valid for coefficients in $\F_p$, and therefore for any field $k$ of characteristic $p$, since $H^n(G,k)\cong H^n(G,\F_p)\otimes_{\F_p} k$.

\vspace{0.2cm} \noindent \textbf{Commutative algebras:} (In this topic, all algebras are assumed to be commutative.) As it can be seen in the table, two references were provided for this case. This was made, because it was proved independently by two groups of authors: Avramov \& Vigu\'{e}-Poirrier (1992) and the Buenos Aires Cyclic Homology Group (1994). One difference between their results is that the latter assumed the characteristic of the ground field to be zero while the former did not. Another notable aspect is that these articles were published more than 10 years prior to the statement of Han's conjecture. So, now we provide a few comments on how precisely Han's property can be deduced from them.

Basically, the following theorem (which is not restricted to finite-dimensional algebras) was proved:

\begin{theorem}
A finitely generated commutative algebra $A$ is smooth if, and only if, its Hochschild homological dimension is finite.
\label{teo:comuta}
\end{theorem}

Now, we will outline that smoothness for finitely generated 
algebras implies in finite global dimension -- and even more for artinian algebras: it is equal to zero. This, together with the theorem, proves Han's property for commutative finitely generated algebras -- and, in particular, for finite-dimensional ones.

Smooth noetherian algebras are, in particular, regular, cf. \cite[Cor. 9.3.13]{weibel_1994}. This implies that the global dimension coincides with the Krull dimension for these algebras, see \cite[5.94]{lam_lec}. Now, one just need to note that finitely generated algebras have finite Krull dimension -- and artinian algebras have zero dimension. Indeed, the Krull dimension of $k[x_1,\hdots, x_n]/I$ is no bigger than $n$ for any ideal $I$. Therefore, every smooth finitely generated algebra has finite global dimension, which is equal to zero when the algebra is finite-dimensional.

\vspace{0.2cm}
\noindent
\textbf{Exterior algebras and quantum complete intersections:}
These two examples share some properties: for instance, they are both Frobenius and local. Furthermore, the results in the references show that their Hochschild homological dimensions are both infinite. By proposition \ref{prop: hoch_finita}, this implies in infinite global dimension (if the field is perfect). However, this conclusion can be also deduced (cf. example \ref{exe:self}) from the more elementary fact that they are non-semisimple Frobenius algebras. Now, we define these algebras and provide some details in these directions.

Given a vector space $V$ over $k$ with basis $\{e_1,\hdots, e_n\}$, the $k$th component of its exterior algebra can be defined as the following quotient:
\[\Lambda^k(V)\vcentcolon =
\dfrac{V^{\otimes k}}
{\langle e_1\otimes\hdots\otimes e_k -\textrm{sgn}(\sigma) e_{\sigma(1)}\otimes\hdots\otimes e_{\sigma(k)}\mid \sigma \in S_k\rangle},
\]
where $S_k$ denotes the symmetric group. In this manner, we define the \textit{exterior algebra of $V$} to be the graded algebra
$\Lambda(V)\vcentcolon=\oplus_{i=0}^n\Lambda^i(V)$, where the product of two elements is simply given by concatenation the tensor products. It is possible to notice that $\dim_k(\Lambda^i(V))=\binom{n}{i}$ and, therefore, that $\dim_k(\Lambda(V))=2^n$.

One can note that $J=\oplus_{i=1}^n\Lambda^i(V)$ is the unique maximal ideal of $\Lambda(V)$ -- which coincides with its Jacobson radical -- so that $\Lambda(V)$ is a local algebra. Dually, we see that $I_0=\Lambda^n(V)$ is the unique minimal ideal of $\Lambda(V)$, since it is a one-dimensional ideal and, for every $0\neq a\in \Lambda(V)$, there exists some $b\in\Lambda(V)$ such that $0\neq ab\in\Lambda^n(V)$. In this manner, any linear functional 
$\lambda\colon\Lambda(V)\to k$ such that $\lambda(I_0)\neq 0$ satisfies the following property: 
for every ideal $I\neq 0$ of $\Lambda(V)$, we have that $\ker(\lambda)\not\supseteq I$. The existence of such $\lambda$ is equivalent to saying that the exterior algebra is Frobenius, see \cite[3.15]{lam_lec}. Using this -- and that, because of $J\neq 0$, it cannot be semisimple -- we can conclude that the global dimension of $\Lambda(V)$ is infinite indeed.

Now, it is possible to see the same properties are satisfied by \textit{quantum complete intersections}\footnote{One motivation for this terminology is that these algebras are the ``quantum version'' of $k[x,y]/(x^a,y^b)$, which are examples of complete intersections rings in the sense of Commutative Algebra. Here, the word ``quantum'' means that the algebra has a relation of quasi-commutativity. This meaning of ``quantum'' was brought to Algebra with the introduction of quantum groups during the '80s, see \cite{Drinfeld1987}. In some applications, the parameter $q$ is interpreted as Planck's constant.}, i.e. algebras of the form
\[A=\frac{k\langle x, y\rangle}{(x^a,xy-qyx,y^b)}
\]
for some $a,b\geqslant 2$ and $0\neq q\in k$. As the ideal $J=(x,y)\subset A$ can be seen to be the unique maximal ideal of $A$, we conclude that $A$ is local -- and not semisimple. The fact that it is Frobenius may be retrieved from \cite[p.509]{bergh2008quantum}.

One of the most interesting aspects of these algebras is that the case $a=2=b$ provided the first counterexample to Happel's property: in \cite{happel_falso}, these algebras were proven to satisfy $\mathrm{hch.dim}(A)=2$ when $q$ is not a root of unity. However, as proved by Y. Han \cite[Proposition 5]{han}, its homology behaviour turned out to be non-pathological. They can be viewed, thus, as one of the main motivations to adapt Happel's question in order to get the proposition of Han's conjecture. With this in mind, the article of Bergh \& Erdmann \cite{bergh2008quantum} may be viewed as a generalization in two directions. On one hand, they showed that Han's property remains valid for arbitrary $a$ and $b$ and, on the other, that the cohomological dimension is still equal to 2 when (and precisely when) $q$ is not a root of unity.
 
% In [http://dx.doi.org/10.1142/S0219498818502158] it is shown that when q is $a$th root [theo 3.1], $dim HH^n(A)=2n +2 \forall n$ for $a=b$.

%One final remark about the article on quantum complete intersections should be made: even though the referenced theorem assumes the element $q$ not to be a root of unity, the fact that the homology is non-zero in infinite degrees is also valid without this hypothesis, as written in \cite[p.520]{bergh2008quantum}.

Two years later, a class of algebras, which generalizes quantum complete intersections, was also proved to satisfy Han's property by showing that its Hochschild homological dimension is infinite. This class is composed by finitely generated algebras of the form
\[A=\frac{k\langle x_1,\hdots, x_n\rangle}{(f_1,\hdots,f_p)}\textrm{, where }
f_1\in k[x_1],\,\, f_i\in (x_2,\hdots, x_n) \textrm{ for }i\geqslant 2
\]
and the algebra $B={k[x_1]}/{(f_1)}$ is assumed to be not smooth. Note that quantum complete intersections are recaptured by taking $n=2$, $p=3$ and $f_1=x^a$, $f_2=xy-qyx$, $f_3=y^b$. The fact that $k[x]/(x^a)$ ($a\geqslant 2$) is not smooth can be deduced from example \ref{exe:polynomials} and theorem \ref{teo:comuta}, or by simply noting that its global dimension is infinite.

\vspace{0.2cm} \noindent \textbf{The examples of Bergh and Madsen:} 
P. Bergh and D. Madsen published two papers, in 2009 and 2017, showing Han's property for some examples of finite-dimensional algebras. The first one \cite{bergh2009hochschild} gives three examples of graded algebras. Their proof relies on a formula of K. Igusa -- relating the Euler characteristic of relative cyclic homology to the graded Cartan determinant -- which forces them to add the assumption that the characteristic of the ground field is zero. In the second article \cite{trivial_extension}, they prove Han's property for trivial extensions of three different classes of algebras.

Concerning the 2009's paper, we must say that, here, a finite-dimensional $k$-algebra $A$ being ``graded''  means that it has a $\N$-grading $A=\oplus_{i\geqslant 0} A_i$ and its Jacobson radical satisfies $J(A)=\oplus_{i\geqslant 1} A_i$. This is called by some authors a \textit{semisimple $\N$-grading}, or a \textit{non-trivial $\N$-grading}. Furthermore, the subalgebra $A_0\cong A/J(A)$ is assumed to be a product of copies of $k$.

\begin{example}
If $kQ$ is a path algebra, where $Q=(Q_0, Q_1)$ is a quiver with a set of vertices $Q_0$ and a set of arrows $Q_1$, then it has a natural grading given by the length of the paths: $kQ=\oplus_{i\geqslant 0} kQ_i$, where $kQ_i$ is the vector subspace generated by the paths of lenght $i$. 
We shall also write $R_Q$ to denote the ideal generated by the arrows $R_Q=\oplus_{i\geqslant 1} kQ_i$.

\begin{enumerate}
    \item If $Q$ does not have oriented cycles (i.e. $kQ$ is finite-dimensional), then, indeed, $J(kQ)$ coincides with $R_Q$ and $kQ_0$ is the sum of $|Q_0|$ copies of $k$.
    
    \item To obtain quotients with the same properties, we can take an \textit{admissible} ideal $I\subset kQ$, i.e. such that  $R_Q^m\subseteq I\subseteq R_Q^2$ for some $m\geqslant 2$. In this manner, $A=kQ/I$ is finite-dimensional (even if $Q$ has cycles) and $J(A)=R_Q/I$, see \cite[2.12]{assem_skowronski_simson_2006}. In order to preserve the grading of $kQ$, we must also assume that $I$ is \textit{homogeneous}, i.e. its generators are linear combinations of paths of the same length. Thus, $A=kQ/I$ has a semisimple $\N$-grading induced from $kQ$ with $A_0\cong A/J(A)\cong kQ/R_Q\cong k^{\oplus Q_0}$.
    
    \item If $A/J(A)$ is a product of copies of $k$, then, by Wedderburn's Splitting Theorem, we have that $A=A/J(A)\oplus J(A)$. In this manner, $A_0=A/J(A)$, $A_1=J(A)$ and $A_i=0$ for $i\geqslant 2$ provides us a grading as required above if, and only if, $A$ is a radical square-zero algebra (i.e. $J(A)^2=0$).

% It is possible to show that if $J(A)^3=$ and $A$ is split for $k$, then it also has a grading.
\end{enumerate}
\end{example}

With this in mind, we will make some comments about two of the three classes considered in the article. For the first one, since the authors already present its definition, we solely mention that the notion of $N$-Koszul algebras (where $N\geqslant 2$ is an integer) is a direct generalization of the characterization of Koszul algebras given in \cite[Prop. 2.1.3]{koszul}. The ordinary case is retrieved when $N=2$.

The second class of examples is given by quotients $A=kQ/I$ where $I$ is an admissible homogeneous ideal and $Q$ is a quiver with some loop (i.e. an arrow which starts and ends at the same vertex). The fact that they always have infinite global dimension is an instance of what is known as the ``no loops conjecture''. In \cite[4.4, 4.5, 5.5]{IGUSA1990161}, K. Igusa proved the conjecture for any admissible quotient of quiver algebras and for every algebra over an algebraically closed field\footnote{Taking into consideration Gabriel's construction of the quiver of an algebra $A$ (over an algebraically closed field), we say that it has a loop if $\Ext^1_A(S,S)\neq 0$ for some simple module $S$, see \cite[4.1.6]{benson_I} or \cite[section II.3]{assem_skowronski_simson_2006}. Noticeably, if $A$ is a path algebra, this equivalent to saying that its quiver contains a loop.}.
In this manner, Han's property is, again, proved after showing that the Hochschild homological dimension for these algebras is infinite. 

Restricting to local algebras, this gives us the following immediate consequence, which is much stronger than Han's property:
\vspace{6pt}
\begin{corollary}
Assume that $\carac(k)=0$ and $A=kQ/I$ is local, where $I$ is an admissible homogeneous ideal. If $\mathrm{hh.dim}(A)$ is finite, then $A\cong k$.
\end{corollary}
\vspace{-10pt}
\begin{proof}
Since $A$ is local, it follows that $0$ and $1$ are its only idempotents, cf. \cite[19.2]{lam_fc}. Thus, $Q$ has only one vertex. By the above, in order to $\mathrm{hh.dim}(A)$ be finite, we also know that $Q$ cannot have loops. Therefore, $Q$ also does not have arrows.
\end{proof}

Now, let us focus our attention in Bergh and Madsen's second article. It is concerned with the \textit{trivial extension} of a finite-dimensional algebra $A$ by its dual $D(A)\vcentcolon=\Hom_k(A,k)$, considered as a $A$-bimodule. 
This algebra is denoted by $T(A)= A\ltimes D(A)$, its vector space structure is defined to be $A\oplus D(A)$ and its multiplication is given by
\[(a,f)\cdot (b,g)=(ab, ag +fb),\,\,\, a,b\in A,\, f,g\in D(A).
\]
These algebras receive the word ``trivial'' in the name, because they are related with the zero element in the cohomology group $HH^2(A,D(A))$, as it can be seen in \cite[p.312]{weibel_1994}.

One notable feature of these trivial extensions is that they are symmetric algebras \cite[16.62]{lam_lec} with Jacobson radical given by $J(A)\oplus D(A)$. Thus, they are non-semisimple self-injective algebras, so that $\gldim(T(A))=\infty$ for every $A\neq 0$. Therefore, once again it must be shown that their Hochschild homological dimensions are infinite.

Now, we sketch some ideas of the proof. The authors start by giving a presentation of $T(A)$ as an admissible quotient of a path algebra -- where it was necessary to assume the field to be algebraically closed. Using a criteria proved in a joint work with Y. Han \cite[Theorem 3.1]{trivial_extension} -- namely, that $kQ/I$ ($I$ admissible) has infinite Hochschild homological dimension whenever it has a $2$-truncated cycle -- they established Han's property for $T(A)$ if $A$ is either local or self-injective. At the end of the article, the property for $T(A)$ was proved when $A$ is graded by utilizing techniques from 2009's article -- in terms of the graded Cartan determinant.

\vspace{0.2cm} \noindent \textbf{Weyl algebras:} In comparison to the examples above, this class is rather exceptional. The $n$th Weyl algebra $A_n(k)$ (over a field $k$) is a certain infinite-dimensional noetherian noncommutative algebra. Its properties are considerably distinct relative to the field chosen: for example, in zero characteristic, $A_n(k)$ is a simple domain, but this is no longer true when fields of positive characteristic are considered. The global dimension can also measure these kind of differences (see \cite[Corollary 5.3]{weyl_global}):
\[\gldim(A_n(k))=
\begin{cases}
n, \textrm{ if } \carac(k)=0\\
2n, \textrm{ if } \carac(k)>0
\end{cases}.
\]

In \cite{solotar_Weyl}, assuming $\carac(k)=0$, the authors proved Han's and Happel's property for the quantum case of a class that generalizes the first Weyl algebra $A_1(k)$ -- while the ordinary non-quantum case was treated ten years prior \cite{weyl_2003}. More explicitly, Hochschild (co)homology was computed for these algebras, and a criterion determining when its global dimension is finite was given. Summing up, they proved that
\[\gldim(A)<\infty \iff \mathrm{hh.dim}(A)\leqslant 2
\iff \mathrm{hch.dim}(A) \leqslant 2.
\]
When this is the case, it was also shown that most of them satisfy $\gldim(A)=2$.

%About truncated for Happel: Cibils (1998) proved for radical square zero and Locatelli (1999) proved for truncated in char(k)=0, and, finally Han et al. (2007) answered for truncated in general.

\subsection{Preservation of Han's property by Extensions}

Recently, some authors gave contributions for the understanding of Han's conjecture in a distinct way from above. Having in mind, for example, a possible inductive step in order to prove the conjecture, many efforts were given in the following direction: encountering extensions of algebras that preserves Han's property, i.e. pairs of algebras $B\subseteq A$ such that, if $B$ satisfy Han's property, then $A$ also satisfy it. We summarize these in table \ref{table:Han's extensions}. For instance, with such results, one can construct from the previous examples many other algebras satisfying Han's property.

\begin{table}[htb]
\centering
\renewcommand{\arraystretch}{1.4}
\begin{tabular}{|lll|} 
\hline
\multicolumn{1}{|l|}{\textbf{Type of Extension}} & \multicolumn{1}{l|}{\textbf{Assumption over the field}} & \textbf{Reference}              \\ 
\hline
corner algebras                                  & perfect                      & \cite[Theorem 2.21]{Cibils_null}   \\ 
\hline
E-triangular algebras                            & perfect                   & \cite[Corollary 2.22]{Cibils_null}     \\ 
\hline
null-square projective algebras                             & perfect                                            & \cite[Theorem 4.8]{Cibils_null}        \\ 
\hline
bounded                                          & -                                                        & \cite[Theorem 4.6]{Eduardo}        \\ 
\hline
strongly proj-bounded                            & -                                                        & \cite[Corollary 6.17]{Kostia}      \\
\hline
\end{tabular}
\vspace{0.4cm}
\caption{Extensions of finite-dimensional algebras which preserves Han's property. The list is organized in chronological order of the references.}
\label{table:Han's extensions}
\end{table}

\noindent\textbf{Null-square algebras:} In \cite{Cibils_null}, the authors analyse \textit{null-square algebras}, which are constructed using two algebras $A$ and $B$, one $A$-$B$-bimodule $N$ and one $B$-$A$-bimodule $M$. They are of the form
\[\begin{bmatrix}
A & N\\
M & B
\end{bmatrix},
\]
where the matrix multiplications are given by the bimodule structure of $M$ and $N$ and the convention $mn=nm=0$ for all $m\in M, n\in N$.
In this way, the algebra above is an extension of $A\times B$.

%-- in more detail, it is ALMOST the trivial extension\footnote{See \cite[p.312]{weibel_1994} for a definition.} of $A\times B$ by the square-zero ideal $I=M\oplus N$.

\begin{itemize}
    \item If $N=0$, then it is called a \textit{corner algebra}
    \item If $M$ and $N$ are projective bimodules, we call it a \textit{null-square projective algebra}.
\end{itemize}

%They did not made the assumption in prop 2.11 that the field is perfect, and they should: the counterexample of of Keller's is a counterexample to 2.11: $A$ has finite global dimension, but $A\otimes A^{op}$ do not.

Provided the field is perfect, it was proved that, if $A$ and $B$ are finite-dimensional algebras satisfying Han's property, then extensions for both types above also satisfy Han's property. For the case of corner algebras, property \ref{prop:corner} was used in order to reduce its homology to the ones of $A$ and $B$.

\vspace{0.2cm} \noindent \textbf{Bounded and proj-bounded extensions:} An extension of algebras $B\subseteq A$ is said to be \textit{bounded} if:
\begin{enumerate}
    \item $A/B$ is of finite projective dimension as a $B$-bimodule
    \item $A/B$ is a left or right projective B-module.
    \item $A/B$ is tensor-nilpotent over $B$, i.e. $(A/B) ^{\otimes_B n} =0$ for some $n$
\end{enumerate}

After a series of papers \cite{Eduardo_split, Eduardo_jacobi, Eduardo}, Cibils, Lanzilotta, Marcos and Solotar proved that if we have such an extension, then
\[B \textrm{ satisfies Han's property} \iff A \textrm{ satisfies Han's property}
\]
(without the necessity of assuming $A$ or $B$ to be finite-dimensional).
In this manner, given an algebra we may analyse it by associating an easier algebra, and it can be chosen to be either smaller or bigger\footnote{It may seem strange to think that a bigger algebra may be simpler, but, as we have already seen, trivial extensions of self-injective algebras are known to satisfy Han's property even though we do not have an answer for self-injective themselves. Unfortunately, trivial extensions are not bounded usually.}. More recently, this was generalized for ``strongly proj-bounded'' extensions.

The authors also provide some criteria in order to recognize if certain extensions satisfy the last two conditions of the definition, see \cite[Theorems 5.16, 5.20]{Eduardo}. Using them, some interesting examples could be given.

\begin{example}
Suppose that $A$ is an extension of $B=kQ/I$ ($I$ an admissible ideal) given by adding arrows to the quiver $Q$ and some possible relations.
    \begin{enumerate}
        \item The case when only arrows are added -- and no new relations -- was treated previously in \cite{inert_arrows} and can be seen as the motivating example for the development of bounded extensions. In this case, $A$ is isomorphic to the tensor algebra (over $B$) $T_B(N)$ for some projective $B$-bimodule $N$. Hence, this extension satisfies a property stronger than the first two conditions in the definition: $A/B$ is projective as a $B$-bimodule. The last condition is also satisfied when $A$ is finite-dimensional.
        %Eles preservam a prop de Han por argumentos do artigo de split bounded
        
        \item \cite[Example 6.2]{Eduardo} Define $B=kQ$ for the quiver $Q$ below
        \[
        \begin{tikzcd}
        & 2 \arrow[r,"d"] & 3  &   \\
        5\arrow[r,"\mu"] & 1   \arrow[r, "b"] & 4 \arrow[u,"c"] 
        \end{tikzcd}
        \]
        and take the extension $A=k\tilde{Q}/J$, where $\tilde{Q}$ is given by adding the arrow $1\xrightarrow{a}2$ in $Q$ and $J=\langle da-cb\rangle$. It can be proved that this extension is bounded.
        Since $\tilde{Q}$ does not have oriented cycles, one of the criteria cited above guarantees that $A/B$ is tensor-nilpotent. 
        The fact that $A/B$ has finite projective dimension as a $(B\otimes B^{op})$-module follows from proposition \ref{prop: hoch_finita}:
        \[\gldim(B\otimes B^{op})=2\cdot \gldim(B)=2.
        \]
    \end{enumerate}
\end{example}

%Escrever resultado explicitamente
%Cite that if we understand the bigger algebra, then we can also make conclusions abou the small one. Citar trivial extensions
% Dizer que não está restrito a dim. finita

In order to prove their result, the authors used a so-called Jacobi-Zariski long nearly exact sequence, which relates the Hochschild homology (of algebras $B$ and $A$) with the relative Hochschild homology (of $A$ with respect to $B$). When $B\subseteq A$ is bounded, this sequence turns out to be exact (in higher degrees). This permits one to conclude that $HH_n(B)$ and $HH_n(A)$ are isomorphic for big enough values of $n$, see \cite[p.52]{Eduardo}.
In this way, Relative Homology -- a theory introduced by G. Hochschild in 1956 \cite{hoch_relativo} but still little used for associative algebras -- is utilized as a fundamental tool in the proofs.
Actually, the very own definition of strongly proj-bounded extensions -- for which, now, we turn our attention -- is made in relative homological terms.

\begin{definition}
An extension $B\subseteq A$ is \textit{strongly proj-bounded} if it satisfies items 1 and 2 from the definition of bounded extensions and, in addition:
\begin{enumerate}
    \setcounter{enumi}{2}
    \item there exists some $p\in\N$ such that $(A/B)^{\otimes_B n}$ is a projective $B$-bimodule for all $n>p$.
    \item $A$, seen as a $A$-bimodule, has finite $B$-relative projective dimension.
\end{enumerate}
\end{definition}

Both conditions above are satisfied if $A/B$ tensor-nilpotent, because $0$ is projective and, as it can be seen in \cite[Proposition 2.3]{Eduardo_jacobi}, there is a $B$-relative projective resolution of $A$ whose length is smaller than $m$ if $(A/B)^{\otimes_B m}= 0$. So, this is, indeed, a generalization of the notion of bounded extensions. In \cite[section 4.2]{Kostia}, examples of strongly proj-bounded extensions of finite-dimensional algebras which are not bounded are presented. Here, we restrict ourselves just to a simpler one.

\begin{example}
\begin{enumerate}
    \item If $B$ is separable (e.g. $B=k$) and $A=B\times B$, then $A/B=B$ is not tensor-nilpotent. However, since $B\otimes B^{op}$ is semisimple, we have that $A/B$ is projective as a $B$-bimodule. Using that $B$-relative projectivity is the same as ordinary projectivity when $B$ is semisimple, we can conclude that $B\subset A$ is strongly proj-bounded.
    
    \item The following example shows that, outside the realm of finite-dimensional algebras, the definition of bounded extensions is much more restrictive. Taking $A=k[x]$ and $B=k$, we conclude once again that $A/B$ is projective as a $B$-bimodule, so that the first three conditions of the last definiton are satisfied. Besides that, we have the following exact sequence
     \[
    0\rightarrow k[x,y]\xrightarrow{\cdot(x-y)} k[x,y]\rightarrow k[x] (\cong k[x,y]/(x-y))\rightarrow 0,
    \]
    so that the projective dimension of $A$ as a $A$-bimodule is $\leqslant 1$. However, this extension is not bounded, since $A/B=(x)$ is not tensor-nilpotent.
\end{enumerate}
\end{example}

% Strongly proj-bounded: 3)There exists a natural number p > 1 (called the index of projectivity) such that the tensor power A/B⊗Bn is projective as a Be -module, for any n > p; 4)A has finite (Ae,Be)-relative projective dimension.

\section{Frontiers of Han's conjecture}
\label{sec:fim}

Having said much about the results already shown towards Han's conjecture, we conclude the article with a few comments on possible future steps.

As noted in section \ref{sec:examples}, many of the examples which were proved to satisfy Han's property are Frobenius: group algebras, exterior algebras, quantum complete intersections, trivial extensions. Therefore, this class of algebras in general seems to be an appealing option to be analysed next. For a more concrete approach, one could start considering some specific cases: for instance, to analyse if the proof for group algebras could be carried out for finite-dimensional Hopf algebras in general. Another possibility would be to focus solely on symmetric algebras while, for a broader setting, self-injective algebras could be chosen as objects of study.

In another aspect, it could be interesting to investigate upper bounds for the realm of algebras satisfying Han's conjecture (which is stated only for finite-dimensional ones). For instance, there are many algebras of finite global dimension whose Hochschild homology is not concentrated in degree zero. For this, one can take Weyl algebras, which were considered above, or even polynomial algebras \cite[Ex. 9.1.3]{weibel_1994}:
\[\gldim(k[x_1,\hdots,x_n])=\mathrm{hh.dim}(k[x_1,\hdots,x_n])=n
\]
In this manner, finite global dimension implying in zero Hochschild homological dimension seems to be a behavior really restricted to finite-dimensional algebras. That said, Han's property (and its converse) is still valid for both examples above.

In \cite{Kostia}, a counterexample to it was given after considering pseudocompact algebras, i.e. topological algebras which are given by an inverse limit of finite-dimensional algebras (considered with discrete topology).

\begin{example} \cite[Remark 6.18]{Kostia} Taking the quiver with infinite vertices below
    \[Q:
    1 \longleftarrow 2 \longleftarrow 3 \longleftarrow \cdots,
    \]
    and the ideal $I=R_Q^2$ generated by paths of lenght two, the pseudocompact algebra $A=k[[Q]]/I$ satisfies $\gldim(A)=\infty$ and $\mathrm{hh.dim}(A)=0$
\end{example}

In opposition to this, it can also be shown that there are certain pseudocompact algebras, obtained from profinite groups, which actually satisfy Han's property, see \cite[Section 4.4]{eu}.

As it can be seen, the example above is not finitely generated nor noetherian. So, this gives a motivation to analyse (if there are any) counterexamples for Han's property in the following classes generalizing finite-dimensional algebras: noetherian, finitely generated, and artinian.

\section{Acknowledgments}
I wish to thank my advisor Kostiantyn Iusenko for his suggestions and support, and for stimulating me and my colleagues for the exchange of ideas. I also express my thanks specially to two of them, Roger R. Primolan and Matheus Schmidt. 
This study was financed in part by the Coordena\c{c}\~ao de Aperfei\c{c}oamento de Pessoal de N\'ivel Superior - Brasil (CAPES) - Finance Code 001.

\bibliographystyle{alphaurl}
\bibliography{Bibliografia.bib}

\end{document}